\documentclass[12pt]{amsart}
\usepackage{graphicx} 
\usepackage{amsmath}

\usepackage[normalem]{ulem}

\usepackage{amsfonts}
\usepackage{amsthm}
\usepackage{color}
\usepackage{mathtools}

\usepackage{hyperref}
\hypersetup{
    colorlinks=true,
    linkcolor=blue
   }

\usepackage[a4paper, margin=1.3in]{geometry}
\usepackage{tikz}
\usepackage{tikz-cd}
\usepackage{enumitem}
\usepackage{xcolor}
\usepackage{amssymb}

\usepackage{chngcntr}

\newtheorem{theorem}{Theorem}[section]
\newtheorem{proposition}[theorem]{Proposition}
\newtheorem{lemma}[theorem]{Lemma}
\newtheorem{corollary}[theorem]{Corollary}
\newtheorem{conjecture}[theorem]{Conjecture}

\theoremstyle{definition}

\newtheorem{definition}[theorem]{Definition}
\newtheorem{remark}[theorem]{Remark}

\newtheorem{question}[theorem]{Question}
\newtheorem{property}[theorem]{Property}
\newtheorem*{property*}{Property}
\newtheorem*{problem}{Problem}
\newcommand\restr[2]{{
  \left.\kern-\nulldelimiterspace
  #1 
  \vphantom{\big|} 
  \right|_{#2}
  }}

\newcommand{\lcm}{\textup{lcm}}

\newcommand{\bigO}{\mathcal{O}}

\DeclareSymbolFont{AMSb}{U}{msb}{m}{n}

\title[Effective non-vanishing for WCI of low codimension]{Effective non-vanishing for weighted complete intersections of low codimension}
\author{Alessandro Passantino }
\date{\today}

\begin{document}

\begin{abstract}
    We show that on quasi-smooth weighted complete intersections of codimension at most 3, any ample Cartier divisor $H$ such that $H-K_X$ is ample admits a nontrivial global section. This is done by proving a generalisation of a numerical conjecture formulated by Pizzato, Sano and Tasin, which relates the existence of global sections of $H$ to the Frobenius number of the numerical semigroup generated by the weights of the ambient projective space.
\end{abstract}

\address{Università degli Studi di Pavia, Dip. di Matematica, Via Ferrata 5, 27100 Pavia, Italy.}
\email{a.passantino@campus.unimib.it}

\maketitle

\thispagestyle{empty}
\section{Introduction}
Throughout the paper, we work over the field $\mathbb{C}$ of complex numbers.

The main goal of this work is to investigate the following conjecture, in the particular case of weighted complete intersections (WCI for short) of general type.

       \begin{conjecture}[Ambro-Kawamata; {\cite[Conjecture 1]{kaw00}}]
       \label{conj:kawamata}
    Let $(X,\Delta)$ be a klt pair, $H$ an ample Cartier divisor such that $H-K_X-\Delta$ is ample. Then, $H^0(X,H) \neq 0$.
    \end{conjecture}

    This conjecture originates from a work on Fano varieties of Ambro \cite{ambro99}; in \cite{kaw00}, where it is first stated, it is proved for varieties of dimension two. To date, the conjecture remains for the most part open, even though a few individual cases have been studied \cite{bh10,horing12,cj16,tasin,jy24}.

    For the present work, we are interested in the setting of \cite{tasin}: there, the authors are able to prove the Ambro-Kawamata conjecture for Fano or Calabi-Yau quasi-smooth weighted complete intersections. For the case of general type, it is proved only in codimension one; codimension two has recently been proved as well in \cite{jy24}.

    The main result of this work is that the conjecture holds for any quasi-smooth weighted complete intersection of codimension at most 3 which is not a linear cone.
    
    \begin{theorem}[=Theorem \ref{thm:main}]
    \label{thm:intromain}
        Let $X \subset \mathbb{P}(a_0,\ldots,a_n)$ be a well-formed quasi-smooth weighted complete intersection which is not a linear cone, such that $\mathrm{codim}X \leq 3$. Then, for any ample Cartier divisor $H$ on $X$ such that $H - K_X$ is ample, $H^0(X, H) \neq 0$.
    \end{theorem}

    In order to prove the statement, we follow the ideas of \cite{tasin} and translate the original conjecture on algebraic varieties into a purely numerical problem. We briefly describe the reasoning behind this approach. Consider a quasi-smooth WCI $X=X_{d_1,\ldots,d_c} \subset \mathbb{P}(a_0,\ldots,a_n)$ of dimension at least 3, defined by polynomials of degrees $d_1,\ldots,d_c$. Many properties of $X$ are determined by the degrees $d_1,\ldots,d_c$ and weights $a_0,\ldots,a_n$: for example, the canonical class $K_X$ or the existence of global sections of a divisor $D\subset X$ depend solely on the degrees and weights of $X$ (a short list of the main properties is given in Section \ref{section:preliminaries}). Furthermore,  it is known that there is a relation between the degrees of $X$ and $a_0,\ldots,a_n$ (Proposition \ref{prop:WCI}). This motivates, in \cite{tasin}, the introduction of $h$-regular pairs (Definition \ref{def:pair}). A $h$-regular pair $(d;a)$ is a pair of tuples of positive integers $d=(d_1,\ldots,d_c)$ and $a=(a_0,\ldots,a_n)$, satisfying properties that generalise the relations between degrees and weights of Proposition \ref{prop:WCI}: in fact, for a quasi-smooth WCI $X$ as before, the degrees and weights give rise to a $h$-regular pair, where $\mathcal{O}_X(h)$ is a positive generator of $\text{Pic}X$. Together with the properties of WCIs, this gives a way to translate the Ambro-Kawamata conjecture on WCIs to a numerical problem on $h$-regular pairs. As an interesting consequence of this approach, we get a deep connection to a classical problem in the theory of numerical semigroups, called the Frobenius problem. When $h=1$ (corresponding, for example, to smooth WCIs), we investigate the following.\\

\begin{conjecture}[{\cite[Conjecture 4.8]{tasin}; Conjecture \ref{conjecture:frob}}]
\label{conj:introfrob}
Let $(d;a)$ be a regular pair, $d=(d_1,\ldots,d_c)$, $a=(a_0,\ldots,a_n)$ and $\delta(d;a)= \sum_{i=1}^c d_i - \sum_{j=0}^n a_j$. Suppose that $ c \leq n$ and $a_i \neq 1$ for any $i$. Then,
$$\delta(d;a) \geq F(a_0,\ldots,a_n),$$
where $F(a_0,\ldots,a_n)$ is the Frobenius number of $a_0,\ldots,a_n$.
\end{conjecture}

Given natural numbers $m_1,\ldots,m_k$ such that $\gcd(m_1,\ldots,m_k)=1$, the Frobenius number $F(m_1,\ldots,m_k)$ of $m_1,\ldots,m_k$ is the largest integer that cannot be written as a non-negative integral combination of $m_1,\ldots,m_k$; the Frobenius problem aims to compute $F(m_1,\ldots,m_k)$. Despite its apparent simplicity, a close formula for $F(m_1,\ldots,m_k)$ only exists for $k=2$, as for $k>2$ it is known that there is no general polynomial relation between $m_1,\ldots,m_k$ and $F(m_1,\ldots,m_k)$ (Theorem \ref{thm:noformula}), and the problem is known to be computationally hard; as a consequence, rather than computing $F(m_1,\ldots,m_k)$, in many cases it is more feasible to consider upper bounds on $F(m_1,\ldots,m_k)$. Then, the interest in Conjecture \ref{conj:introfrob} (and its generalisation to $h$-regular pairs, Conjecture \ref{conjecture:h-frob}) is twofold: on one hand, a positive answer would complete the proof of Conjecture \ref{conj:kawamata} for all quasi-smooth WCIs; on the other, it would improve known bounds on $F(m_1,\ldots,m_k)$ such as the one due to Brauer (Proposition \ref{prop:br}; also see \cite[Section 4.1]{tasin}). Because of this, we focus on studying $h$-regular pairs and prove the following.

\begin{theorem}[=Corollary \ref{cor:maincor}]
\label{thm:mainpairs}
Let $(d;a)=(d_1,\ldots,d_c;a_0,\ldots,a_n)$ be a $h$-regular pair such that $c \leq 3$, $c \leq n$ and $a_i \nmid h$ for all $0 \leq i \leq n$. Then, $$\delta(d;a) \geq F^h(a_0,\ldots,a_n);$$ here, $F^h(a_0,\ldots,a_n)$ is the largest multiple of $h$ that cannot be written as a non-negative integral combination of $a_0,\ldots,a_n$.
\end{theorem}
In particular, we deduce that both Conjecture \ref{conj:introfrob} and its generalisation to $h$-regular pairs (Conjecture \ref{conjecture:h-frob}) hold for pairs with $c\leq 3$, and as a corollary, Theorem \ref{thm:intromain} follows.

The proof of Theorem \ref{thm:mainpairs} is done case-by-case, but relies on a collection of new results that either give recursive bounds on Frobenius numbers, or reduce the number and structure of the pairs that need to be considered. For the most part, these are independent from the setting of Theorem \ref{thm:mainpairs}: for instance, the bounds on Frobenius numbers do not require the use of pairs, and the reductions are independent from the number of degrees of the pair being considered. The only exception is given by the following.

\begin{theorem}[=Corollary \ref{cor:noasterisk}]
Conjecture \ref{conj:introfrob} and its generalisation to $h$-regular pairs (Conjecture \ref{conjecture:h-frob}) are equivalent for pairs with at most 3 degrees.
\end{theorem}
While we are not able to prove that this happens in general, we show that it is the case for a wide range of pairs (Lemma \ref{lemma:reduction}), which motivates our expectation that this result should hold for pairs with any number of degrees. Another type of reduction that is shown to hold in many cases (though not used throughout the paper) can be found in the last section, where we investigate whether it is enough to prove Conjecture \ref{conj:introfrob} for pairs with maximal number of degrees.

Finally, we point out that the approach of this paper to proving Theorem \ref{thm:intromain} differs from both \cite{tasin} and \cite{jy24} in codimension 1 and 2, as neither of them uses $h$-regular pairs to prove their results in the general type case. So, Theorem \ref{thm:mainpairs} is the first result to provide evidence in support of the validity of Conjecture \ref{conj:introfrob} for pairs of any codimension. 
\\

\noindent \textbf{Acknowledgements.} The author would like to thank Luca Tasin for many important discussions on the content and exposition of the results, and Taro Sano for pointing out a number of improvements to the clarity of the paper.

\section{Preliminaries}
\label{section:preliminaries}
In the following, we will often mix sheaf and divisorial notation. First of all, we review some definitions in order to introduce quasi-smooth weighted complete intersections. An in-depth examination of these objects and their properties can be found in \cite{ianofl}, \cite{dolgachev}.

\subsection{Weighted projective spaces}

\begin{definition}
    Let $x_0,\ldots,x_n$ be affine coordinates on $\mathbb{A}^{n+1}$, $a_0,\ldots,a_n$ positive integers. Consider the action of $\mathbb{C}^*$ given by

    $$\lambda(x_0,\ldots,x_n)=(\lambda^{a_0}x_0,\ldots,\lambda^{a_n}x_n).$$
The quotient $\mathbb{P}(a_0,\ldots,a_n) \coloneqq (\mathbb{A}^{n+1}\setminus \{0\}) / \mathbb{C}^*$ is a projective variety called the \emph{weighted projective space} of weights $a_0,\ldots,a_n$.
\end{definition}

\begin{definition}
    A weighted projective space $\mathbb{P}(a_0,\ldots,a_n)$ is \emph{well-formed} if $$\gcd(a_0,\ldots,\hat{a_i},\ldots,a_n)=1$$ for any $0 \leq i \leq n$, that is, any set of $n$ weights does not share a common non-trivial factor.
\end{definition}

\begin{remark}
\label{rmk:wps}
\hfill
        \begin{enumerate}[label=(\roman*)]
        \item{Due to the $\mathbb{C}^*$-action, any well-formed weighted projective space which is not the standard projective space is singular. By construction, the singular locus of the weighted projective space $\mathbb{P}(a_0,\ldots,a_n)$ is the union of the strata 
        $$\Pi_I \coloneqq  \{x_i = 0 \mid i \notin I\},$$
        where $I = \{i_1,\ldots,i_k\} \subset \{a_0,\ldots,a_n\}$ is any subset such that $a_I \coloneqq \gcd_{i\in I} (a_i) >1$.}
        
    \item{For any weighted projective space $\mathbb{P}$, its class group $\text{Cl}(\mathbb{P})$ is cyclic; on the other hand, since $\mathbb{P}$ is singular, in general it does not coincide with $\text{Pic}(\mathbb{P})$.}
    \item{
    Generalising the standard projective space, the canonical sheaf of $\mathbb{P}$ is given by \cite[Corollary 6B.8]{br86}
    $$K_\mathbb{P} = \bigO_{\mathbb{P}}(-\sum_{i=0}^{n}a_i).$$
    }

\end{enumerate}

\end{remark}

\subsection{Weighted complete intersections}

\begin{definition}
\hfill
\begin{itemize}
    \item{
    Let $X$ be a variety in $\mathbb{P}(a_0,\ldots,a_n)$, and $I$ its homogeneous ideal. Suppose that $I$ is generated by a regular sequence $\{f_i\}$ of homogenous polynomials such that $\deg f_i = d_i$. Then, we say that $X$ is a \emph{weighted complete intersection} (WCI for short) of multidegree $(d_1,\ldots,d_c)$. Any WCI of multidegree $d_1,\ldots,d_c$ will be denoted by $X_{d_1,\ldots,d_c}$.
    }
    \item{
    A WCI $X_{d_1,\ldots,d_c}$ is called a \emph{linear cone} if $a_i=d_j$ for some $i$ and $j$.
    }
\end{itemize}
\end{definition}

Since weighted projective spaces are usually singular, it should be expected that subvarieties are rarely smooth. This leads to the following, weaker definition.

\begin{definition}[\cite{dimca}, Definition 1]
\hfill
\begin{itemize}
    \item{A subvariety $X$ of the weighted projective space $\mathbb{P}=\mathbb{P}(a_0,\ldots,a_n)$ is said to be  \emph{well-formed} if $\text{codim}_X(X \cap \text{Sing}(\mathbb{P}))\geq 2$.}
    \item{
    Let $\pi \colon \mathbb{A}^{n+1}\setminus \{0\} \rightarrow \mathbb{P}(a_0,\ldots,a_n)$ be the canonical projection. We say that $X$ is \emph{quasi-smooth} if the punctured affine cone $\pi^{-1}(X)$ is smooth.
    }
\end{itemize}
\end{definition}

Note that if $P$ is a singular point of $\pi^{-1}(X)$, all points in the same fiber of $\pi$ are also singular. As a consequence, any singularity on a quasi-smooth WCI $X$ only appears due to the $\mathbb{C}^*$-action. 

\begin{remark}
    In \cite{ianofl}, a different definition of well-formedness is used: there, a WCI $X \subset \mathbb{P}$ of codimension $c$ is well-formed if it does not contain any codimension $c+1$ singular stratum of $\mathbb{P}$. As pointed out in \cite{prz}, the two definitions are not equivalent (as the one in \cite{ianofl} is weaker), but they do coincide when $\dim{X}\geq 3$ and $X$ is not a linear cone.
\end{remark}

\begin{property}
\label{property:WCI}
    Let $X=X_{d_1,\ldots,d_c} \subset \mathbb{P}=\mathbb{P}(a_0,\ldots,a_n)$ be a well-formed quasi-smooth WCI. 
    \begin{enumerate}[label=(\roman*)]
        \item{$\text{Sing}(X) = X \cap \text{Sing}(\mathbb{P})$.
        }
        \item{If $\dim X >2$, then $\text{Cl}(X) \cong \mathbb{Z}$ and is generated by $\bigO_X(1) \coloneqq \restr{\bigO_\mathbb{P}(1)}{X}$.}
    
        \item{Adjunction holds: in particular, if $\dim X >2$ the canonical sheaf of $X$ is given by

        $$K_X = \bigO_X(\sum_{i=1}^{c} d_i - \sum_{j=0}^n a_i).$$
        The integer number $\delta = \sum_{i=1}^{c} d_i - \sum_{j=0}^n a_i$ is called the \emph{amplitude} of X.
        }
    
        \item{
        The space of global sections of $\bigO_X(k)$ can be computed from the homogenous coordinate ring of $X$. More precisely, let $A = \mathbb{C}[x_0,\ldots,x_n]/(f_1,\ldots,f_c)$ be the homogenenous coordinate ring of $X$, $A_k$ its $k$-graded part. Then,

        $$H^0(X,\bigO_X(k)) \simeq A_k.$$
        }

\end{enumerate}
\end{property}

\begin{remark}
    For the most part, we will suppose that the WCI we consider are \emph{not} linear cones: in fact, a general WCI $$X=X_{d_1,\ldots,d_c} \subset \mathbb{P}(a_0,\ldots,a_n)$$ with $d_1 = a_0$ is isomorphic to 
    $$X'=X'_{d_2,\ldots,d_c} \subset \mathbb{P}(a_1,\ldots,a_n).$$ 
\end{remark}

In \cite{tasin}, it was proved that there are necessary and sufficient numerical conditions for a general WCI $X$ which is not a linear cone to be quasi-smooth. While we will not use the result in its general form, it gives a complete generalisation of previous partial results such as \cite[Chapter 8]{ianofl} and \cite[Proposition 2.3]{chen15}.

\begin{proposition}[{\cite{tasin}, Proposition 3.1}]
\label{prop:WCI}
Let $X = X_{d_1,\ldots,d_c} \subset \mathbb{P}(a_0,\ldots,a_n)$ be a quasi-smooth WCI which is not a linear cone. For a subset $I=\{i_1,\ldots,i_k\} \subset \{0,\ldots,n\}$ let $\rho_I = \min \{c,k\}$, and for a $k$-tuple of natural numbers $m=(m_1,\ldots,m_k)$ write $m \cdot a_I \coloneqq \sum_{j=0}^k m_j a_{i_j}.$
Then, one of the following conditions holds.

\begin{itemize}
    \item[(Q1)]{
    There exist distinct integers $p_1,\ldots,p_{\rho_I} \in \{1,\ldots,c\}$ and $k$-tuples $M_1,\ldots,M_{\rho_I} \in \mathbb{N}^k$ such that $M_j \cdot a_I = d_{p_j}$ for $j=1,\ldots,\rho_I$.
    }
    \item[(Q2)]{
    Up to a permutation of the degrees, there exist:
    \begin{itemize}
        \item {
        an integer $l < \rho_I$,
        }
        \item{integers $e_{\mu,r}\in \{0,\ldots,n\}\setminus I$ for $\mu = 1,\ldots,k-l$ and $r=l+1,\ldots,c$,
        }
        \item{$k$-tuples $M_1,\ldots,M_l$ such that $M_j \cdot a_I = d_j$ for $j=1,\ldots,l$,
        }
        \item{for each $r$, $k$-tuples $M_{\mu,r}$, $\mu=1,\ldots,k-l$ such that $a_{e_{\mu,r}}+M_{\mu,r}\cdot a_I = d_r$,
        }
        \end{itemize}
        satisfying the following property: for any subset $J \subset \{l+1,\ldots,c\}$, 
        $$|\{e_{\mu,r} \colon r\in J, \mu=1,\ldots,k-l\}| \geq k-l+|J|-1.$$
    }
\end{itemize}

Conversely, if for all subsets $I \subset \{0,\ldots,n\}$ either $(Q1)$ or $(Q2)$ holds, then a general WCI $X_{d_1,\ldots,d_c} \subset \mathbb{P}(a_0\,\ldots,a_n)$ is quasi-smooth.

\end{proposition}

\subsection{$h$-regular pairs}
\label{section:pairs}
The main takeaway from Proposition \ref{prop:WCI} is that quasi-smoothness of a general WCI $X$ can be checked solely on degrees, weights and their numerical relations. At the same time, quasi-smoothness implies strong constraints on the degrees and weights; the following corollary is a display of this fact, and leads to the definition of $h$-regular pairs.

\begin{corollary}[{\cite[Proposition 3.6]{tasin}}]
\label{corollary:h-regular}
Let $X = X_{d_1,\ldots,d_c} \subset \mathbb{P}(a_0,\ldots,a_n)$ be a quasi-smooth well-formed WCI which is not a linear cone. Suppose $\textup{Pic}(X)$ is generated by $\bigO(h)$, $h>0$. For any subset $I = \{i_1,\ldots,i_k\} \subset \{0,\ldots, n\}$ such that $a_I= \gcd(a_{i_1},\ldots,a_{i_k})>1$, one of the following holds.

\begin{enumerate}[label=(\roman*)]
    \item{There exist distinct integers $p_1,\ldots,p_k$ such that $a_I \mid d_{p_1},\ldots,d_{p_k}$.}
    \item{$a_I \mid h.$}
\end{enumerate}
    
\end{corollary}

\begin{definition}
\label{def:pair}
Let $c, n \in \mathbb{N}$, $d_1, \ldots, d_c$, $a_0, \ldots, a_n \geq 1$ be natural numbers, and write $(d;a) = (d_1,\ldots,d_c ; a_0,\ldots,a_n)$. Let $\bar c = \{1,\ldots,c \}$, $\bar n = \{0, \ldots, n \}$. We say that the pair $(d;a)$ is \emph{h-regular} for a positive integer $h$ if, for any non-empty subset $I = \{i_1,\ldots, i_k\} \subset \bar n$ such that $a_I \coloneqq \gcd(a_{i_1},\ldots, a_{i_k})>1$, at least one of the following holds.

\begin{enumerate}[label=(\roman*)]
    \item{There exist $k$ distinct integers $p_1,\ldots,p_k \in \bar c$ such that $a_I \mid d_{p_1},\ldots,d_{p_k}$.}
    \item{$a_I \mid h$.}
\end{enumerate}

If $h=1$, we say the pair is \emph{regular}. In analogy to the geometrical setting, we call the integers $d_i$ the \emph{degrees} of the pair and $c$ the \emph{codimension}, the integers $a_j$ the \emph{weights} and $n$ the \emph{dimension}; also, we say that $h$ is the \emph{regularity index} of the pair.
\end{definition}

Note that the regularity index is not unique, as if a pair $(d;a)$ is $h$-regular, it is also $h'$-regular for any $h'>h$ such that $h \mid h'$. While a minimal index exists, in some cases it is useful to allow $h$ to be not minimal (for example, see Lemma \ref{lemma:h-prop}).

\begin{remark}
    By Property \ref{property:WCI} and Corollary \ref{corollary:h-regular}, if $(d;a)$ are degrees and weights of a well-formed quasi-smooth WCI $X$ which is not a linear cone, and $\dim X > 2$, then $(d;a)$ is a $h$-regular pair, where $\bigO(h)$ is a positive generator of $\text{Pic}X$. On the other hand, it is very easy to find $h$-regular pairs which do not come from a well-formed quasi-smooth WCI.
\end{remark}

We now fix some of the notation we use to work with $h$-regular pairs.

Let $(d;a)$ be a $h$-regular pair.
\begin{itemize}
\item{We write $|d|=c$ (resp. $|a|=n$) if $d\in \mathbb{N}^c$ (resp. $a \in \mathbb{N}^n$). For integers $d_i$ and $a_j$, we write $d_i \in d$ (resp. $a_j \in a)$ if $d_i$ appears in $d$ (resp. $a_j$ appears in $a$).}
\item{For a pair $(d;a)$ with $|d|=c$, $|a|=n$, we define $$\delta(d;a) \coloneqq \sum_{i=0}^c d_i - \sum_{j=0}^n a_j,$$ and call it the amplitude of the pair. If the pair comes from a well-formed quasi-smooth WCI $X$ of dimension $>2$, then $K_X \simeq \bigO_X(\delta(d;a))$ by Property \ref{property:WCI}(iv). When the pair is clear from the context, we will simply write $\delta$ for $\delta(d;a)$.}
\item{Let $g$ be a positive integer. Write $I_g = \{i \in \bar n \colon g \mid a_i \}$, $J_g = \{j \in \bar c \colon g \mid d_j \}$. We can construct two new types of pairs from $(d;a)$:
\begin{itemize}
\item{$(d^g; a^g)$, where $d^g = ((d_j / g)_{j \in J_g}, (d_j)_{j \in \bar c \setminus J_g})$, $a^g = ((a_i / g)_{i \in I_g}$, $(a_i)_{i \in \bar n \setminus I_g})$, is obtained by dividing all divisible degrees and weights by $g$;}
\item{$(d(g), a(g))$, where $d(g) = (d_j)_{j \in J_g}$, $a(g) = (a_i)_{i \in I_g}$, is obtained by only considering degrees and weights divisible by $g$.}
\end{itemize}
When the pair is clear from the context, we will write $\delta(g) = \delta(d(g);a(g))$ and $\delta^g = \delta(d^g;a^g)$.
}

\end{itemize}

\begin{lemma}[{\cite[Lemmas 4.5 and 4.6]{tasin}}]
\label{lemma:h-prop}
\hfill
\begin{itemize}
\item{Let $(d;a)$ be a $h$-regular pair and $p$ a prime not dividing $h$. Then the pairs $(d(p);a(p))$, $(d(p)/p; a(p)/p)$ and $(d^p;a^p)$ are $h$-regular.}
\item{Let $(d;a)$ be a $h$-regular pair and $p$ a prime dividing $h$. Then $(d(p);a(p))$ is $h$-regular, while $(d(p)/p; a(p)/p)$ and $(d^p;a^p)$ are $h/p$-regular.}
\end{itemize}
\end{lemma}

\begin{corollary}
\label{cor:h-prop}
    Let $(d;a)$ be a $h$-regular pair, and $g=\gcd(a_{i_1},\ldots,a_{i_m})>1$ for some weights $a_{i_1},\ldots,a_{i_m}$. Then, $(d(g);a(g))$ is $h$-regular and $(d(g)/g;a(g)/g)$ is $h/g$-regular.
\end{corollary}

\begin{proof}
    Let $g=\prod p_i^{k_i}$, with $p_i>1$ prime numbers. We prove the statement about $(d(g);a(g))$ by induction on $k=\sum k_i$; the base case $k=1$ is Lemma \ref{lemma:h-prop}. Now suppose the statement is true up to $\sum k_i=k-1$. By the same Lemma, the pair $(d';a')=(d(p_1)/p_1;a(p_1)/p_1)$ is $h$-regular if $p_1 \nmid h$, or $h/p_1$-regular if $p_1 \mid h$. Since $p_1 \mid g$, $a_{i_1}/p_1,\ldots,a_{i_m}/p_1 \in a'$ and $g'=\gcd(a_{i_1}/p_1,\ldots,a_{i_m}/p_1)=g/p_1$. By induction, $(d'(g');a'(g'))$ is either $h$-regular or $h/p_1$ regular as before. Since $(d(g);a(g))=p_1(d'(g');a'(g'))$, we get that $(d(g);a(g))$ is $h$-regular.

    The statement on $(d(g)/g;a(g)/g)$, follows directly from the fact that $(d(g);a(g))$ is $h$-regular, and repeatedly applying Lemma \ref{lemma:h-prop} for every prime $p$ dividing $h$.
\end{proof}

\section{Effective non-vanishing and the Frobenius problem}
We can now introduce the ideas to approach Conjecture \ref{conj:kawamata} for a quasi-smooth WCI from the point of view of $h$-regular pairs. The properties described in Section \ref{section:preliminaries} show that Conjecture \ref{conj:kawamata} for quasi-smooth weighted complete intersections is equivalent to the following.

\begin{conjecture}
\label{conj:WCI}
Let $(d;a)=(d_1,\ldots,d_c; a_0,\ldots,a_n)$ be the pair of degrees and weights of a well-formed quasi-smooth WCI $X$ which is not a linear cone with $\dim X > 2$, and $\bigO(h)$ a positive generator of $\textup{Pic}X$. Let $\delta(d;a)$ be the amplitude of $X$. Then, for any positive integer $k$ such that $h \mid k$ and $k > \delta(d;a)$ there exist natural numbers $x_0,\ldots,x_n$ such that 
$$\sum_{i=0}^n x_i a_i = k.$$
\end{conjecture}

In \cite{tasin}, Conjecture \ref{conj:WCI} was proved for Fano and Calabi-Yau well-formed quasi-smooth WCIs which are not linear cones. In these cases, since $H-K_X$ is automatically ample, the conjecture is the same as stating that $h$ belongs to the semigroup generated by $a_0,\ldots,a_n$. This is done in the more general setting of $h$-regular pairs, for which the following stronger fact is proved.

\begin{proposition}[{\cite[Proposition 5.12, Corollary 5.13]{tasin}}]
\label{prop:delta>0}
Let $(d;a)$ be a $h$-regular pair. if $a_i \nmid h$ for any $i=0,\ldots,n$, then $\delta(d;a)>0$. Equivalently, if $\delta(d;a)\leq 0$ then there exists a weight $a_i$ such that $a_i \mid h$.
\end{proposition}

Similarly, we would like to prove a statement on $h$-regular pairs that implies Conjecture \ref{conj:WCI} for weighted complete intersections of general type. In order to precisely state the problem we will consider, we need two definitions.

\begin{definition}
\label{def:frob}
Let $m_1,\ldots,m_k>1$ be coprime natural numbers and $S=\langle m_1,\ldots,m_k \rangle$ the monoid generated by $m_1,\ldots,m_k$ by addition. Then, the set $G(S) = \mathbb{N} \setminus S$ is finite, so we can define the \emph{Frobenius number}
of $m_1,\ldots,m_k$ (or equivalently, of $S$) as
$$F(S)=F(m_1,\ldots,m_k)=\max (\mathbb{Z}\setminus S).$$
\end{definition}
In other words, the Frobenius number of $m_1,\ldots,m_k$ is the largest integer which cannot be written as a non-negative linear combination of $m_1,\ldots,m_k$. When $m_i \neq 1$ for every $i$, $F(S) \in G(S)$; otherwise, $F(S) = -1$.

\begin{definition}
Let $m_1,\ldots,m_k$ as in Definition \ref{def:frob}. The \emph{$\frac{1}{h}$-Frobenius number}\footnote{As far as the author is aware, in the literature this concept usually does not have a distinct name. Our choice has two reasons: first, this notion is closely related to quotients of numerical semigroups by an integer \cite[Chapter 5]{rosales}, usually written $\frac{S}{k}$ for a numerical semigroup $S$ and positive integer $k$; in this setting, the literature simply calls $F^k(a_0,\ldots,a_n)/k$ the Frobenius number of such quotients (and writes $F(S/k)$), as it is unambiguous. In our case, we work at the same time with different values of $h$ so we want to make the notation clearer. Secondly, the more natural name of $h$-Frobenius number is already common in literature to refer to the generalised Frobenius number (written $F_h(a_0,\ldots,a_n)$), which is an unrelated concept.}
of $m_1,\ldots,m_k$ (or of $S$) is defined as
$$F^h(S) = F^h(m_1,\ldots,m_k) = \max (h\mathbb{Z} \setminus (h\mathbb{Z}\cap S)).$$
\end{definition}
Not much differently from before, the $\frac{1}{h}$-Frobenius number is the largest multiple of $h$ which cannot be written as a combination of $m_1,\ldots,m_k$, and if $m_1,\ldots,m_k \nmid h$ then $F^h(S) \in G(S)$, otherwise $F^h(S) = -h$.

By abuse of notation, we still talk about the Frobenius number (or $\frac{1}{h}$-Frobenius number) when $g=\gcd(m_1,\ldots,m_k)>1$, even though it is, in general, not well defined; in that case, we mean the following: if $\gcd(g,h)=1$, then
$$F^h(m_1,\ldots,m_k) = gF^h(m_1/g,\ldots,m_k/g).$$ while if $g \mid h$, 
$$F^h(m_1,\ldots,m_k) = g F^{h/g}(m_1/g,\ldots,m_k/g).$$
In general, let $G=\gcd(g,h)$, then
$$F^h(m_1,\ldots,m_k) = gF^{h/G}(m_1/g,\ldots,m_k/g).$$

Given these definitions, we are ready to generalise Conjecture \ref{conj:WCI} to $h$-regular pairs. This was first stated in \cite{tasin}, but only the regular case ($h$=1) was considered. Here, we study the $h$-regular case as well, and remove the redundant or unnecessary hypotheses used in the original statement.

\begin{conjecture}[{\cite[Conjecture 4.8]{tasin}}]
\label{conjecture:frob}
Let $(d_1,\ldots,d_c; a_0,\ldots,a_n)$ be a regular pair such that $ c \leq n$ and $a_i \neq 1$ for any $i$. Then,
$$\delta(d;a) \geq F(a_0,\ldots,a_n).$$
\end{conjecture}
Note that by definition, in this case we have that $\gcd(a_0,\ldots,a_n)=1$.

\begin{conjecture}
\label{conjecture:h-frob}
Let $(d_1,\ldots,d_c; a_0,\ldots,a_n)$ be a $h$-regular pair, such that $ c \leq n$ and $a_i \nmid h$ for any $i$. Then, 
$$\delta(d;a) \geq F^h(a_0,\ldots,a_n).$$
\end{conjecture}

Before moving to a brief examination of the problem of computing $F(m_1,\ldots,m_s)$, it is worth remarking three facts about Conjectures \ref{conjecture:frob} and \ref{conjecture:h-frob}.
\hfill
\begin{itemize} 
\item{By Proposition \ref{prop:delta>0}, under the hypotheses of Conjectures \ref{conjecture:frob} and \ref{conjecture:h-frob}, $\delta(d;a)>0$ so we are only taking the general type case into consideration.}
\item{For both conjectures, we ask the condition $c \leq n$ as in the original statement of \cite{tasin}; note that, since the description of $\text{Pic}X$ as a cyclic group in Property \ref{property:WCI} only holds when $\dim X \geq 3$, Conjectures \ref{conjecture:frob} and \ref{conjecture:h-frob} correctly generalise Conjecture \ref{conj:WCI} only when $\dim X \leq n-3$, that is $c \leq n-3$. Still, we expect that using this weaker assumption on $c$ does not change the validity of the conjecture.}
\item{On a similar note, we usually allow pairs with $d_i=a_j$ for some $i,j$ even if the problem on weighted complete intersections avoids the linear cone case. In fact, if for a pair $(d;a)$ we have $\delta(d;a) \geq F^h(a)$, then adding a pair of identical weights and degrees does not change the inequality: the amplitude of the pair remains the same and the $\frac{1}{h}$-Frobenius number does not increase.}
\end{itemize}
\subsection{About the Frobenius problem}

Given a set of coprime natural numbers $a_1,\ldots,a_n$, the problem of computing its Frobenius number is known in literature as the \emph{Frobenius problem}.

\begin{problem}[Frobenius problem]
Let $S = \langle a_1,\ldots,a_n \rangle$ be a numerical semigroup. Compute $F(S)$, the Frobenius number of $S$.
\end{problem}

Despite the apparent simplicity, it turns out that the computation of $F(S)$ is hard under many points of view. On the positive side, an explicit answer is known for $n=2$: in fact, $F(a_1,a_2)=[a_1,a_2]-a_1-a_2$, where $[a_1,a_2]=\lcm(a_1,a_2)$; interestingly, for larger $n$ it is known that there is no general polynomial relation between a set of coprime natural numbers and their Frobenius number.

\begin{theorem}[{\cite{cur90}}]
\label{thm:noformula}
    Let $$\text{$A=\{(a_1,a_2,a_3) \in \mathbb{N}^3 \mid a_1 < a_2 < a_3,a_1,a_2$ are prime and $a_1,a_2 \nmid a_3\}$}.$$
    Then, there is no nontrivial polynomial $P \in \mathbb{C}[X_1,X_2,X_3,Y]$ such that $$P(a_1,a_2,a_3,F(a_1,a_2,a_3))=0$$ for all $(a_1,a_2,a_3) \in A$. In other words, it is not possible to find a polynomial relation between $a_1,a_2,a_3$ and $F(a_1,a_2,a_3)$ holding for all semigroups of given embedding dimension $n$ (that is, the cardinality of the minimal set of generators). 
\end{theorem}
Even from a computational point of view, the Frobenius problem is known to be hard, since it is NP-hard (the state of the art in the known computational aspects and algorithms can be found in \cite{alfonsin}).

Since neither formulas nor algorithms easily give the answer to the problem, bounds on $F(S)$ are of large interest. For instance, the following result is one of the best known upper bounds. It also computes exactly the Frobenius number of a semigroup if the generators form a sequence with a particular structure.

\begin{proposition}[\cite{br42}, \cite{br54}]
\label{prop:br}
Let $(a_1,\ldots,a_n)$ be coprime positive numbers, and let $g_j=\gcd(a_1,\ldots,a_j).$ Then,
 $$F(a_1,\ldots,a_n) \leq \sum_{j=2}^n \frac{g_{j-1}}{g_j}a_j - \sum_{i=1}^n a_i.$$
Equality holds if and only if $a_1,\ldots,a_n$ form a \emph{telescopic sequence}: if we write $S_{i-1}$ for the semigroup generated by $a_1/g_{i-1},\ldots,a_{i-1}/g_{i-1}$, then $a_i/g_i \in S_{i-1}$ for all $i=2,\ldots,n$.

\end{proposition}

As pointed out in \cite[Section 4.1]{tasin}, this bound has a natural relation with regular pairs: in fact, write $d_j=\frac{g_{j-1}}{g_{j}}a_j$, then the pair $(d_2,\ldots,d_{n};a_1,\ldots,a_n),$ is regular and its amplitude is precisely the bound of Proposition \ref{prop:br}. In general, this pair does not achieve the minimal value of $\delta$ for a given set of weights. For example, the pair $$(d;a)=(6p,6q,pq;2p,3p,2q,3q)$$ with $p,q$ primes large enough satisfies $$\delta(d;a)=F(2p,3p,2q,3q)=pq+p+q,$$ but the sequence $2p,3p,2q,3q$ is not telescopic (independently from how it is ordered), hence $\delta(d;a) \neq F(2p,3p,2q,3q)$. Thus, in general, for a given set of weights $a_0,\ldots,a_n$, Conjecture \ref{conjecture:frob} gives a better bound on $F(a_0,\ldots,a_n)$ than Proposition \ref{prop:br}.

\section{Properties of $h$-regular pairs and Frobenius numbers}
\subsection{Properties of Frobenius numbers}
The following is a simple observation; for the sake of readability, even though it will be used multiple times, we will not explicitly refer to it.

\begin{lemma}
\label{lemma:frobineq}
Let $a_0,\ldots,a_n$ and $a_0'$ be positive integers such that $a_0 \mid a_0'$. Suppose that $g\coloneqq\gcd(a_0,\ldots,a_n)=\gcd(a_0',a_1,\ldots,a_n)$. Then, for any $h>0$,
    $$F^h(a_0',a_1,\ldots,a_n) \geq F^h(a_0,\ldots,a_n).$$.
\end{lemma}

\begin{proof}
First, consider the case $g=1$.
Since $a_0 \mid a_0'$,
    $$\langle a_0',\ldots,a_n \rangle \subset \langle a_0,\ldots,a_n \rangle$$
    and the statement follows from the definition of $\frac{1}{h}$-Frobenius number.

For the general case, write $G=\gcd(g,h)$. Since by definition
$$F^h(a_0',\ldots,a_n) = gF^{h/G}(a_0'/g,\ldots,a_n/g)$$
and
$$F^h(a_0,\ldots,a_n) = gF^{h/G}(a_0/g,\ldots,a_n/g),$$
as $\gcd(a_0'/g,\ldots,a_n/g)=\gcd(a_0/g,\ldots,a_n/g)=1$, the statement follows as before.

\end{proof}

\begin{corollary}
\label{cor:multiples}
If $a_0,\ldots,a_n$ are positive integers, then for any positive integer $m>0$ such that $\gcd(a_0,\ldots,a_k, ma_{k+1},\ldots,ma_n)=\gcd(a_0,\ldots,a_k, a_{k+1},\ldots,a_n)$ for some $k\geq 0$, we have that for any $h >0$,
$$F^h(a_0,\ldots,a_k,ma_{k+1},\ldots,ma_n) \geq F^h(a_0,\ldots,a_n).$$
\end{corollary}

Next, we have a classical result that, in some cases, relates the Frobenius number of a set of weights to the Frobenius number of a set of smaller weights.

\begin{lemma}[{\cite[Lemma 3.1.7]{alfonsin}}]
\label{lemma:degred}
Let $a_0,\ldots,a_n$ be positive integers such that $\gcd(a_0,\ldots,a_n)=1$, and let $g=\gcd(a_0,\ldots,a_{n-1})$. Then,

$$F(a_0,\ldots,a_n) = g F(\frac{a_0}{g},\ldots,\frac{a_{n-1}}{g},a_n) + (g-1) a_n.$$

\end{lemma}

\subsection{Reduction to degrees with no common factors}
We first need a lower bound on $\delta(d;a)$ for regular pairs satisfying the hypotheses of Conjecture \ref{conjecture:frob}.

\begin{lemma}[{\cite[Proposition 5.2]{tasin}}]
\label{lemma:deltabound}
Let $(d;a)=(d_1,\ldots,d_c;a_0,\ldots,a_n)$ be a regular pair such that $a_i \neq 1$ for any $i$. Then, $\delta(d;a) \geq c$.
\end{lemma}

We can now show that in the regular case of Conjecture \ref{conjecture:frob}, we can suppose that there is no non-trivial factor dividing all the degrees.

\begin{lemma}
\label{lemma:primered}

Let $(d;a)=(d_1,\ldots,d_c; a_0,\ldots,a_n)$ be a regular pair such that $c \leq n$. Let $g \coloneqq \gcd(d_1,\ldots,d_c) > 1$ and $p \mid g$  a prime dividing $g$. Suppose $p \mid a_0,\ldots,a_k$, $p \nmid a_{k+1},\ldots,a_n$. If 
$$\delta^p=\delta(d^p;a^p) \geq F(a^p)=F(\frac{a_0}{p},\ldots,\frac{a_k}{p},a_{k+1},\ldots,a_n),$$ then 
$$\delta(d;a) \geq F(a_0,\ldots,a_n)$$ also holds.
\end{lemma}

\begin{proof}
Note that by regularity, $k+1 \leq c \leq n$. Then, the statement follows from Corollary \ref{cor:multiples} and Lemma \ref{lemma:degred}:

\begin{equation*}
\begin{split}
\delta(d;a)& = p\delta^p+(p-1)\sum_{i=k+1}^n a_i \geq p\delta^p+(p-1)a_n \\ 
&\geq pF(\frac{a_0}{p},\ldots,\frac{a_{k}}{p},a_{k+1},\ldots,a_n)+(p-1)a_n \\
&=F(a_0,\ldots,a_{k-1},pa_k,\ldots,pa_{n-1},a_n)\geq F(a_0,\ldots,a_n).
\end{split}
\end{equation*}
\end{proof}

\begin{proposition}
    Suppose that Conjecture \ref{conjecture:frob} holds for any regular pair $(d^*;a^*)$ such that $\gcd(d_1^*,\ldots,d_c^*)=1$, then it also holds for any pair $(d;a)$ such that $\gcd(d_1,\ldots,d_c)>1$.
\end{proposition}

\begin{proof}
    Let $g=\gcd(d_1,\ldots,d_c)>1$, and write $g = \prod p_i^{k_i}$ for prime numbers $p_i$ and integers $k_i>0$. We show the statement by induction on $k=\sum k_i$; the base case $k=0$ is given by hypothesis, so suppose that the statement holds for any pair $(d^{**};a^{**})$ such that $\gcd(d_1^{**},\ldots,d_c^{**})=\prod p_i^{k_i'}$ and $\sum k_i' \leq k-1$. 

Consider the pair $(d^p;a^p)=(d_1',\ldots,d_c';a_0',\ldots,a_n')$ where $p$ is a prime dividing $g$, obtained by dividing all divisible degrees and weights of $(d;a)$ by $p$; then, $g'=\gcd(d_1',\ldots,d_c')=g/p$. 
    \begin{itemize}
        \item{If $a_i' \neq 1$ for any $i$, $(d^p;a^p)$ satisfies the hypotheses of Conjecture \ref{conjecture:frob}, hence $$\delta^p=\delta(d^p;a^p) \geq F(a_0',\ldots,a_n')=F(a^p)$$ by the induction step. Then, $\delta(d;a) \geq F(a_0,\ldots,a_n)$ by Lemma \ref{lemma:primered}.}
        \item{If $ a_i'=1$ for some $i$, then we may assume that $a_0,\ldots,a_m=p$ and $a_{m+1},\ldots,a_n \neq p$, where $m \leq c-1$ by regularity. Now, the subpair $(d^p;a^{p,[m]})=(d_1',\ldots,d_c'; a_{m+1}',\ldots,a_n')$ (obtained by only considering the weights of $(d^p;a^p)$ which are not equal to $1$) satisfies the hypotheses of Lemma \ref{lemma:deltabound}, therefore $\delta(d^p;a^{p,[m]})\geq c$. Since $a_0'=\ldots=a_m'=1$ and $m\leq c-1$, this implies that $\delta(d^p;a^p)=\delta(d';a')\geq 0$; on the other hand, $F(a_0',\ldots,a_n')=-1$ because some weight is equal to 1, hence $\delta(d^p;a^p) > F(a_0',\ldots,a_n')$, and the statement follows from Lemma \ref{lemma:primered}.}
        \end{itemize}
\end{proof}

\subsection{Recursive bounds on Frobenius numbers}

We now move to proving a family of recursive bounds on Frobenius numbers, which will be useful in taking an inductive approach on the number of weights of pairs.

\begin{lemma}
\label{lemma:recfrob1g}
    Let $a_0,\ldots,a_n$ be coprime positive integers, $g=\gcd(a_0,\ldots,a_k)$, $G$ a positive integer coprime with $g$. Then,
    $$F(a_0,\ldots,a_n) \leq F^g(a_0,\ldots,a_k)+F^G(a_{k+1},\ldots,a_n,g)+gG.$$
\end{lemma}

\begin{proof}
First of all, note that under the assumptions, $\gcd(g,a_{k+1},\ldots,a_n)=1$, so that $F^G(a_{k+1},\ldots,a_n,g)$ is well defined. Let $N > F(a_0,\ldots,a_k) + F^G(a_{k+1},\ldots,a_n,g)+gG$, then
$$N - F(a_0,\ldots,a_k) - F^G(a_{k+1},\ldots,a_n,g)-g-G > gG-g-G.$$
Since $F(g,G)=gG-g-G$, by definition we get that there exist $y_1, y_2 \geq 0$ such that
$$N-F(a_0\ldots,a_k)-F^G(a_{k+1},\ldots,a_n,g)-g-G = y_1g+y_2G,$$
and reordering the terms,
$$N - (F(a_0,\ldots,a_k)+(y_1+1)g) = F^G(a_{k+1},\ldots,a_n,g)+(y_2+1)G.$$
Again by definition, since 
$$F^G(a_{k+1},\ldots,a_n,g)+(y_2+1)G=\sum_{i=k+1}^n x_ia_i + yg$$
with $x_i, y \geq 0$, we get
$$N=F(a_0,\ldots,a_k)+(y+y_1+1)g+\sum_{i=k+1}^n x_ia_i,$$
hence by definition,
$$N = \sum_{j=0}^k x_j a_j +  \sum_{i=k+1}^n x_i a_i$$
for some $x_j \geq 0$.
\end{proof}

\begin{lemma}
\label{lemma:recfrob1}
    Let $a_0,\ldots,a_n$ be coprime positive integers, $g=\gcd(a_0,\ldots,a_k)$ for $k \geq 0$. Then,
    $$F(a_0,\ldots,a_n) \leq F^g(a_0,\ldots,a_k)+F(a_{k+1},\ldots,a_n,g)+g.$$
\end{lemma}

\begin{proof}
This is Lemma \ref{lemma:recfrob1g} when $G=1$.

\end{proof}

\begin{lemma}
\label{lemma:recfrob2g}
Let $a_0,\ldots,a_n$ be coprime positive integers, and consider (not necessarily disjoint) non-empty subsets $I_1,\ldots, I_k \subset \{0,\ldots,n\}$. Let $g_j=\gcd_{i \in I_j}(a_i)$ and write $a_{I_j}$ for the set of weights indexed by $I_j$. Suppose the $g_j$ are coprime. Then,
$$F(a_0,\ldots,a_n) \leq F(g_1,\ldots,g_k) + \sum_{j=1}^k g_j +\sum_{j=1}^k F^{g_j}(a_{I_j}).$$
\end{lemma}

\begin{proof}
Let
$$N>F(g_1,\ldots,g_k) + \sum_{j=1}^k g_j +\sum_{j=1}^k F^{g_j}(a_{I_j}),$$
then,
$$N-\sum_{j=1}^k g_j -\sum_{j=1}^k F^{g_j}(a_{I_j})>F(g_1,\ldots,g_k)$$
and by definition,
$$N-\sum_{j=1}^k g_j -\sum_{j=1}^k F^{g_j}(a_{I_j})=\sum_{i=1}^k y_ig_i.$$
We can rewrite this as
$$N= \sum_{j=1}^k (F^{g_j}(a_{I_j})+(y_j+1)g_j).$$
By definition for each Frobenius number, we get
$$N= \sum_{j=1}^k \sum_{l \in I_j} x_l a_l,$$
for $x_l \geq 0$.
\end{proof}

\subsection{Reduction to the regular case}
\label{section:reduction}
While at first glance Conjectures \ref{conjecture:frob} and \ref{conjecture:h-frob} are not equally strong statements, it turns out that in most cases the $h$-regular conjecture can be reduced to the regular case. In the following, we make heavy use of the notation introduced in Section \ref{section:pairs} to streamline the statements in the results and proofs.

\begin{lemma}
\label{lemma:reduction}
Suppose Conjecture \ref{conjecture:h-frob} holds for $h'$-regular pairs of codimension at most $c$, such that $h' < h$. Then, the conjecture also holds for any $h$-regular pair $(d;a)$ such that $|d|=c$, satisfying at least one of the following conditions.
\begin{enumerate}[label=(\roman*)]
\item{\label{lemma:deltabarp<0}
There is a prime $p$ dividing $h$ such that $\bar \delta(p) \leq 0$, where $\bar \delta(p)=\delta(d;a)-\delta(p)$.}

\item{\label{lemma:deltabarp>0}
There is a prime $p$ dividing $h$ such that $\bar \delta(p) \geq 0$ and $|d(p)| < |a(p)|$.}

\item{\label{lemma:deltabarg<0}
There is a greatest common divisor $g = \gcd(a_{i_1},\ldots,a_{i_k})>1$ such that $\delta(d;a)\geq \delta(d(g);a(g))$ and $|d(g)|<|a(g)|$.}

\end{enumerate}
\end{lemma}

\begin{proof}
\hfill
\begin{enumerate}[label=(\roman*)]
\item{First, note that $\delta(d;a)=p\delta^p - (p-1)\bar \delta(p)$; since $\bar \delta(p)\leq 0$, then $$\delta(d;a)\geq p\delta^p.$$ By Lemma \ref{lemma:h-prop}, $(d^p;a^p)$ is $h/p$-regular; $(d^p;a^p)$ still satisfies the conditions of Conjecture \ref{conjecture:h-frob}, hence by hypothesis $p\delta^p \geq pF^{h/p}(a'_0,\ldots,a'_n)$, where $a'_i$ are the weights of $(d^p;a^p)$. But $$pF^{h/p}(a'_0,\ldots,a'_n) = F^h(pa'_0,\ldots,pa'_n) \geq F^h(a_0,\ldots,a_n)$$ by construction, hence
$$\delta(d;a) \geq p\delta^p \geq pF^{h/p}(a_0',\ldots,a_n')\geq F^h(a_0,\ldots,a_n).$$}

\item{
By Lemma \ref{lemma:h-prop}, the pair $(d(p)/p;a(p)/p)$ is $h/p$-regular, and by hypothesis $|a(p)| > |d(p)|$, so it satisfies the assumptions of Conjecture \ref{conjecture:h-frob}. Then, $\delta(p)/p \geq F^{h/p}(a(p)/p)$, which implies that 
$$\delta(d;a) \geq \delta(p) \geq pF^{h/p}(a(p)/p) = F^h(a(p))\geq F^h(a_0,\ldots,a_n).$$}

\item{The proof follows as in the previous case: by Lemma \ref{lemma:h-prop} and Corollary \ref{cor:h-prop}, $(d(g);a(g))$ is $h$-regular and $(d(g)/g; a(g)/g)$ is $h/g$-regular. Then, as before, $\delta(d;a) \geq \delta(g) \geq F^h(a(g))$.}
\end{enumerate}
\end{proof}

It follows that whenever a $h$-regular pair $(d;a)$ satisfies any of the conditions of Lemma \ref{lemma:reduction}, Conjecture \ref{conjecture:h-frob} holds for $(d;a)$ if it is true for any $h'$-regular pair with $h'<h$ and codimension at most $|d|$. If any $h$-regular pair always satisfied one of the conditions Lemma \ref{lemma:reduction}, it would mean that Conjectures \ref{conjecture:frob} and \ref{conjecture:h-frob} are equivalent. We are then led to investigate whether there exists a $h'$-regular pair $(d';a')$ which does not satisfy any of the conditions of Lemma \ref{lemma:reduction}, that is, if there exists a $h$'-regular pair $(d';a')$ satisfying the following property.
\begin{property*}[$*$]
For any integer $k$, let $d'(k)$ (resp. $a'(k)$) be the subset of degrees (resp. weights) of $(d';a')$ divisible by $k$. Then both of the following hold.
\begin{itemize}
\item{For every $k=p$ a prime dividing $h$, $\bar \delta'(p) > 0$ and $|d'(p)| \geq |a'(p)|$.}
\item{For every $k=\gcd(a_{i_1},\ldots,a_{i_l})>1$ such that $|d'(k)|<|a'(k)|$, $\delta(d';a')<\delta(d'(k);a'(k))$.}
\end{itemize}
\end{property*}

Although it is not clear whether a pair satisfying Property $(*)$ exists, in low codimensions we can show that there is no such pair.

\begin{proposition}
\label{prop:noasterisk}
Let $(d;a)$ be a $h$-regular pair, $|d|\in\{1,2,3\}$, satisfying the hypotheses of Conjecture \ref{conjecture:h-frob}. Then, $(d;a)$ does not satisfy Property $(*)$.
\end{proposition}
\begin{proof}

When $|d|=1,2$ the statement is straightforward. In fact, if $|d|=1$  Lemma \ref{lemma:reduction}\ref{lemma:deltabarp<0} is automatically satisfied for any $p \mid h$; when $|d|=2$, either \ref{lemma:deltabarp<0} or \ref{lemma:deltabarp>0} is guaranteed to hold, because for any $p \mid h$ either $|d(p)|=1 < |a(p)|$ or $|d(p)|=2$ (hence $\bar \delta(p) \geq 0$). So we consider the case $|d|=3$.

Without loss of generality, let $g=\gcd(a_0,\ldots, a_k)>1$ be such that $|d(g)|<|a(g)|$. Note that we can suppose $|d(g)|=1$ and $|a(g)|=2$:
\begin{itemize}
\item{$|d(g)|=0$: this cannot happen, as no $a_i$ divides $h$;}
\item{$|d(g)|=2$: since $|d(g)|<|a(g)|$, then $3\leq |a(g)| \leq |a(p)|$ for any $p \mid g$. Then, either case \ref{lemma:deltabarp<0} or \ref{lemma:deltabarp>0} holds (if $|d(p)|=3$ or $|d(p)|=2$, respectively).}
\item{$|d(g)|=3$: then $|d(p)|=3$, and we have case \ref{lemma:deltabarp<0}.}
\item{$|d(g)|=1,|a(g)|=3$: $|a(p)|\geq 3$ for all $p \mid g$, hence either case \ref{lemma:deltabarp<0} or \ref{lemma:deltabarp>0} holds.}
\end{itemize}
Then suppose $|d(g)|=1, |a(g)|=2$; we can also assume that for any $p \mid g$, $a(g)=a(p)$, otherwise $|a(p)|\geq 3\geq |d(p)|$ and we are again in case \ref{lemma:deltabarp<0} or \ref{lemma:deltabarp>0}. Without loss of generality, under these assumptions we have $g \mid a_0,a_1$ and

\begin{equation*}
\begin{cases}
d_1= g \cdot k \\
d_2 = g_1 \cdot k_1 \\
d_3 = g_2 \cdot k_2 \\
\end{cases}
\end{equation*}
where $g= g_1 g_2$ for $g_1,g_2 >1$, $\gcd(g_1,g_2)=1$, and $k,k_1,k_2 \geq 1$.
If Property $(*)$ holds, we know that $\delta - \delta(g)<0$, which means that $d_2+d_3-\sum_{i=2}^n a_i <0$; on the other hand, since $\delta - \delta(p)>0$ for any $p \mid g_1$, we get that $d_3-\sum_{i=2}^n a_i > 0$, which is a contradiction.

Thus, Property $(*)$ cannot hold.
\end{proof}

\begin{corollary} 
\label{cor:noasterisk}
Conjectures \ref{conjecture:frob} and \ref{conjecture:h-frob} are equivalent for $h$-regular pairs of codimension at most 3.
\end{corollary}

\section{Effective non-vanishing}

We can now prove Conjecture \ref{conj:WCI} when $\text{codim}(X)\leq 3$.
\begin{theorem}
\label{thm:main}
    Let $X \subset \mathbb{P}$ be a well-formed quasi-smooth WCI which is not a linear cone, $\textup{codim}X \leq 3$ and $H$ an ample Cartier divisor such that $H-K_X$ is ample. Then, $|H| \neq \emptyset$.
\end{theorem}

As in the Fano and Calabi-Yau case of \cite{tasin}, this is done by proving the conjecture in the more general setting of $h$-regular pairs. Thanks to Proposition \ref{prop:noasterisk}, we only need to consider the regular case, that is Conjecture \ref{conjecture:frob}. 
\\

\begin{theorem}[cf. {\cite[Proposition 6.2]{tasin}}]
\label{thm:codim1}
Let $(d;a)=(d_1;a_0,\ldots,a_n)$ be a $h$-regular pair, $n \geq 1$, and suppose $a_i \nmid h$ for all $i$. Then, $\delta(d;a) \geq F^h(a_0,\ldots,a_n).$
\end{theorem}

\begin{proof}
For pairs coming from quasi-smooth weighted complete intersections, the statement was already proved in \cite[Proposition 6.2]{tasin}. In general, we notice that it follows directly from Proposition \ref{prop:noasterisk} by reducing to the regular case: in fact, for a regular pair $(d;a)=(d_1;a_0,\ldots,a_n)$ of codimension 1, it is easy to see that Conjecture \ref{conjecture:frob} holds, as all weights must be pairwise coprime, hence $d_1 \geq \prod a_i$ and in particular $\delta(d;a) \geq F(a_i,a_j)$ for all $a_i,a_j \in a$.
\end{proof}

For pairs of codimension 1, we can actually obtain a stronger result, which gives an estimate of $F(a_0,\ldots,a_n)$.

\begin{theorem}
\label{thm:codim1-2}
    Let $(d;a)=(d_1;a_0,\ldots,a_n)$ be a $h$-regular pair, $n \geq 1$, and suppose $a_i \nmid h$ for all $i$. Then, $\delta(d;a) \geq F(a_0,\ldots,a_n).$
\end{theorem}

\begin{proof}
    We can assume that $\gcd(a_0,\ldots,a_n)=1$. We proceed by induction on $n$. The case $n=1$ follows from Proposition \ref{thm:codim1},
     so suppose that the statement holds for any pair of dimension at most $n-1$. Note that for every weight $a_i$ there exists a prime $p_i$ and a positive integer $l_i$ such that $p_i^{l_i} \mid a_i$ but $p_i^{l_i} \nmid a_j$ for all $j \neq i$, otherwise we would have that $a_i \mid h$. Because of this, we deduce that $[a_0,\ldots,a_{n-1}] \coloneqq \lcm(a_0,\ldots,a_{n-1}) \neq d_1$ and $[a_n,g]=ga_n \neq d_1$, where $g=\gcd(a_0,\ldots,a_{n-1})$. Then, consider the pairs $(d';a')=([a_0,\ldots,a_{n-1}]; a_0,\ldots,a_{n-1})$ and $(d'';a'')=([a_n,g];a_n,g)$. Since both satisfy the induction hypotheses, we have that $\delta(d';a')\geq F(a_0,\ldots,a_{n-1})$ and $\delta(d'';a'') \geq F(a_n,g)$. Then,
     $$\delta(d;a)\geq \delta(d';a')+\delta(d'';a'')+g \geq F(a_0,\ldots,a_n)$$
     by Lemma \ref{lemma:recfrob1}.
\end{proof}

\begin{theorem}
\label{thm:codim2}
Let $(d;a) = (d_1,d_2; a_0,\ldots,a_n)$, $n \geq 2$ be a $h$-regular pair such that $a_i \nmid h$ for every $i$. Then, $\delta(d;a) \geq F^h(a_0,\ldots,a_n)$.
\end{theorem}

\begin{proof}

For pairs coming from weighted complete intersections, this was proved in \cite[Theorem 1.2]{jy24}. In general, by Proposition \ref{prop:noasterisk}, we can suppose that $(d;a)$ is a regular pair. Then, by Lemma \ref{lemma:primered} we can reduce to the case of $(a_i,a_j)=1$ for every $0\leq i,j \leq n$, $i \neq j$. Up to a permutation of the weights, suppose $a_0,\ldots,a_k \mid d_1$, $a_{k+1},\ldots,a_n \mid d_2$, with $k>1$. Both the pairs $(d';a') = (d_1;a_0,\ldots,a_{k})$ and $(d'';a'')=(d_2;a_{k+1},\ldots,a_n)$ are regular of codimension 1, thus $\delta(d'';a'')>0$. Then, by Proposition \ref{thm:codim1}, $$\delta(d;a) \geq \delta(d';a') \geq F(a_0,\ldots,a_k) \geq F(a_0,\ldots,a_n).$$
\end{proof}

\begin{theorem}
\label{thm:codim3}
Let $(d;a) = (d_1,d_2,d_3; a_0,\ldots,a_n)$ be a $h$-regular pair such that $n\geq 3$ and $a_i \nmid h$ for all $i$. Then, $\delta(d;a) \geq F^h(a_0,\ldots,a_n)$.
\end{theorem}

\begin{proof}

We first set the assumptions on the pairs we consider, and define the notations that are used in the proof.
As in codimension 1 and 2, we can assume that $(d;a)$ is regular by Proposition \ref{prop:noasterisk}. We can suppose that $d_i \neq a_j$ for any $i,j$ (otherwise this reduces to the case of codimension 2) and by Lemma \ref{lemma:primered} we can restrict to the case $\gcd(a_{i_1},a_{i_2},a_{i_3})=1$ for all distinct $i_1,i_2,i_3$. 
For any degree $d_j\in d$, let  $A_j=\{a_i\in a \colon  a_i\mid d_j\}$ be the set of weights dividing $d_j$. Define the pairs $(d';a')=(d_2,d_3;a_l,\ldots,a_n)$ where $a_i,\ldots,a_{l-1} \in A_1$ and $a_l,\ldots,a_n \notin A_1$, and if $|A_j|=2$, let $(d'';a'')=(d_2,d_3;a_2,\ldots,a_n,g)$, where $g=\gcd(a_0,a_1)$. By our assumptions, both $(d';a')$ and $(d'';a'')$ are $d_1/m$-regular, where $m=\lcm\{a_i \in A_1\}$: in fact, since no three weights share a common factor, any common factor of two weights among $a_l,\ldots,a_n$ cannot divide $m$; for $(d'';a'')$, as $(d;a)$ is regular, $g$ also divides either $d_2$ or $d_3$. We write $\delta'=\delta(d';a')$ and $\delta''=\delta(d'';a'')$. Also, $\delta(d';a')>0$ by Proposition \ref{prop:delta>0} because $a_l,\ldots,a_n \nmid d_1$.

We will use the convention that if $a_i \mid a_{i'}$ for some $i,{i'}$, then $a_i$ and $a_{i'}$ belong to distinct sets $A_j$ and $A_{j'}$, that is, the set $A_j$ give a partition of the weights of $a$. More precisely, even though $a_i$ and $a_{i'}$ must belong to at least one common $A_j$, since there must be another $A_{j'}$ such that $a_i \in A_{j'}$, we will say that $a_i \in A_{j'}$ and $a_{i'} \in A_j$, but $a_i \notin A_j$ and $a_{i'} \notin A_{j'}$. Note that the regularity of $(d';a')$ and $(d'';a'')$ is unchanged by this convention.
\\

We prove the statement by induction on $k=\min_j \{|A_j| \colon |A_j|>1\}$, and consider different cases for every $k$. Note that such a $k$ exists, because the pair is regular and since $|a|>|c|$ there must be two weights dividing the same degree. Without loss of generality, we will always assume that $k=|A_1|$, and that the weights belonging to $A_1$ are $a_0,\ldots,a_{k-1}$.
We first prove the statement under the assumption that $d_1=[a_0,\ldots,a_{k-1}]$, where $[a_0,\ldots,a_{k-1}]=\lcm(a_0,\ldots,a_{k-1})$, then show how the proof generalises to the (easier) case $d_1 >[a_0,\ldots,a_{k-1}]$.

$\boldsymbol{\textbf{Case }k=2.}$
\begin{itemize}

    \item{$\boldsymbol{g=\gcd(a_0,a_1)=1}$: then,
    $$\delta(d;a) = (d_1-a_0-a_1)+\delta'=([a_0,a_1]-a_0-a_1) + \delta' > F(a_0,a_1),$$
    and we are done because $a_0$ and $a_1$ are coprime.}

    \item{$\boldsymbol{g>1}$: write $\delta(d;a)=(d_1-a_0-a_1)+\delta(d'';a'')+g$. Since $|a''|=n>|d''|=2$, by Proposition \ref{thm:codim2} 
    $$\delta(d'';a'')\geq F(a_2,\ldots,a_n,g),$$
    therefore 
            \begin{equation*}
            \begin{split}
            \delta(d;a) &\geq ([a_0,a_1]-a_0-a_1)+\delta(d'';a'')+g \\
                        &\geq F(a_0,a_1)+F(a_2,\ldots,a_n,g)+g \geq F(a_0,\ldots,a_n)
            \end{split}
            \end{equation*}
    by Lemma \ref{lemma:recfrob1}}.
\end{itemize}

$\textbf{Case }\boldsymbol{k=3.}$

\begin{itemize}

    \item{$\boldsymbol{[a_0,a_1]=[a_0,a_2]=[a_1,a_2]=[a_0,a_1,a_2]=d_1}$: in this case, since
    $$[a_0,a_1,a_2]=\frac{a_0a_1a_2}{\gcd(a_0,a_1)\gcd(a_0,a_2)\gcd(a_1,a_2)}$$
    and
    $$[a_i,a_j]=\frac{a_ia_j}{\gcd(a_i,a_j)},$$
    we get
    \begin{equation*}
        \begin{cases}
            &a_0=\gcd(a_0,a_1)\gcd(a_0,a_2)\\
            &a_1=\gcd(a_0,a_1)\gcd(a_1,a_2)\\
            &a_2=\gcd(a_0,a_2)\gcd(a_1,a_2)\\
        \end{cases}
     \end{equation*}
    Two among $\gcd(a_0,a_1),\gcd(a_0,a_2),\gcd(a_1,a_2)$ (which are $\neq 1$ by the convention on the weights) must divide one of $d_2$ or $d_3$, hence one of $a_0,a_1,a_2$ actually divides $d_2$ or $d_3$ as well, say $a_2$. Then, $(d'';a'')$ is regular and the proof follows as in the case $k=2$.}

    \item{$\boldsymbol{[a_0,a_1] \neq d_1:}$ we consider three subcases.

        \begin{itemize}
            \item{$\boldsymbol{g=\gcd(a_0,a_1)=1}$: then $d_1 \geq [a_0,a_1]+a_2$ because $a_2 \neq d_1$ by assumption, hence 
            \begin{equation*}
            \begin{split}
                \delta(d;a) &=(d_1-a_0-a_1-a_2)+\delta(d';a') \\ 
                            &\geq [a_0,a_1]-a_0-a_1 +\delta(d';a') > F(a_0,a_1).
            \end{split}
            \end{equation*}
                }

            \item{$\boldsymbol{g>1, ga_2 \neq d_1}$: in this case, $d_1 \geq [a_0,a_1]+ga_2$ and $g,a_2$ are coprime, hence
            \begin{equation*}
            \begin{split}
            \delta(d;a) &\geq ([a_0,a_1]-a_0-a_1)+(ga_2-a_2-g)+g +\delta(d';a') \\
                        &> F(a_0,a_1)+F(a_2,g)+g \geq F(a_0,a_1,a_2) \geq F(a_0,\ldots,a_n)
            \end{split}
            \end{equation*}
    
            by Lemma \ref{lemma:recfrob1}.}
\\

            \item{$\boldsymbol{g>1, ga_2=d_1}$: we first show that we can reduce to a pair with $n=3$. In fact, consider the pair $(d^*;a^*)=(d_2,d_3;a_3,\ldots,a_n,g)$. This pair is regular and if $|a^*|=n-1>2$, by Proposition \ref{thm:codim2} $\delta(d^*;a^*)\geq F(a_3,\ldots,a_n,g)$. Since $d_1 \geq [a_0,a_1]+a_2$, we get
            $$\delta(d;a)>F(a_0,a_1)+F(a_3,\ldots,a_n,g)+g \geq F(a_0,\ldots,a_n)$$ by Lemma \ref{lemma:recfrob1}. Then, suppose $n=3$ and write $g_i=a_i/g$ for $i=0,1$. We consider three more subcases.
            
            \begin{itemize}

                \item{$\boldsymbol{g,a_3 \mid d_2}$: since $d_1\geq [a_0,a_1]+a_2$,
                \begin{equation*}
                    \begin{split}
                    \delta(d;a) &> [a_0,a_1]-a_0-a_1+(ga_3-g-a_3)+g \\
                    &=F(a_0,a_1)+F(a_3,g)+g,
                    \end{split}
                \end{equation*} 
                $$$$
                and we get the result from Lemma \ref{lemma:recfrob1}.}

                \item{$\boldsymbol{a_3 \mid d_2, g\mid d_3, g_0\mid d_3}$: since $a_0=gg_0$, $a_0 \mid d_3$ and we can conclude by induction by noticing that
                $$\delta(d;a)=\delta(d_1;a_1,a_2)+\delta(d_2,d_3;a_0,a_3,g)+g$$
                and $(d_2,d_3;a_0,a_3,g)$ is regular, hence we can use Lemma   \ref{lemma:recfrob1}.}

                \item{$\boldsymbol{g_0,g_1,a_3 \mid d_2$, $g = d_3}$: the pair $(d^*;a^*)=(d_1,d_2;a_1,a_2,a_3)$ is still regular, and by induction
                $$\delta(d^*;a^*)\geq F(a_1,a_2,a_3).$$           
                Since $\delta(d;a)=\delta(d^*;a^*)+d_3-a_0$ and $d_3=g \mid a_0$, then $\delta(d;a)\geq F(a_1,a_2,a_3).$}
            \end{itemize}
        }
                
        \end{itemize}
    }
\end{itemize}

\textbf{Case } $\boldsymbol{k>3}.$
    \begin{itemize}

        \item{\textbf{All weights dividing $\boldsymbol{d_1}$ are pairwise coprime}: then, $(d_1;a_0,\ldots,a_{k-1})$ is regular and the statement follows directly from Proposition \ref{thm:codim1}.}

        \item{$\boldsymbol{\gcd(a_0,a_1)=g>1 \textbf{ and }[a_0,a_1] \neq d_1}$: then $d_1 \geq [a_0,a_1]+d_1/g$ and the pair $(d^*;a^*)=(d_1/g,d_2,d_3;a_2,\ldots,a_n,g)$ is again regular. If $k\geq 5$ (which implies $n\geq 4$), then $|a^*|\geq 4$ and we can use induction to say that 
        $$\delta(d^*;a^*)\geq F(a_2,\ldots,a_n,g),$$
        in which case
        $$\delta(d;a) \geq ([a_0,a_1]-a_0-a_1)+\delta(d^*;a^*)+g \geq F(a_0,\ldots,a_n)$$
        by Lemma \ref{lemma:recfrob1}.
        
        Otherwise $k=n+1=4$, and let $g'=\gcd(a_2,a_3)$; note that $[a_0,a_1]\leq d_1/g'$ (and $\neq d_1$ by hypothesis), $[a_2,a_3] \leq d_1/g\neq d_1$, and if $g'=1$ then the statement follows easily, because
        $$\delta(d;a) \geq ([a_2,a_3]-a_2-a_3)+([a_0,a_1]-a_0-a_1)+\delta(d';a') > F(a_2,a_3).$$ 
        Thus, we can suppose $g'>1$, which implies 
        $$[a_0,a_1]+[a_2,a_3]\leq \frac{d_1(g+g')}{gg'}\leq \frac{5}{6}d_1.$$
        If we can show that $gg' \leq \frac{1}{6}d_1$ we are done, because then
        $$d_1 \geq [a_0,a_1]+[a_2,a_3]+gg',$$ and
        $$\delta(d;a)\geq F(a_0,a_1)+F(a_2,a_3)+gg',$$
        hence $\delta(d;a)\geq F(a_0,a_1,a_2,a_3)$ by Lemma \ref{lemma:recfrob1g}. Since neither of $a_0$ and $a_1$ divides the other, there must be coprime numbers $q_0,q_1>1$ such that $a_0=gq_0,a_1=gq_1$ and $\gcd(a_2,g)=\gcd(a_3,g)=1$, thus $$gg'=[g,g']\leq \frac{d_1}{q_0q_1} \leq \frac{1}{6}d_1.$$
        Hence, $d_1 \geq [a_0,a_1]+[a_2,a_3]+gg'$ and we get the statement.}

        \item{$\boldsymbol{\gcd(a_0,a_1)=g>1 \textbf{ and }[a_0,a_1]=d_1}$:

         Write $a_0=gg_0$, $a_1=gg_1$ and $q=\gcd(g_0,g)\gcd(g_1,g)$ for $g_0,g_1>1$ such that $\gcd(g_0,g_1)=1$. Then
        $$a_0+a_1+\frac{d_1}{qg}=\frac{d_1}{g_1}+\frac{d_1}{g_0}+\frac{d_1}{qg}=d_1(\frac{1}{g_0}+\frac{1}{g_1}+\frac{1}{qg}).$$
        There are very few values of $g,g_0,g_1$ satisfying the previous assumptions and such that 
        $$\frac{1}{g_0}+\frac{1}{g_1}+\frac{1}{qg}>1,$$
        but they force at least one of $g_0$ and $g_1$ to be prime: for the inequality to hold, each of $g_0$, $g_1$ and $qg$ must be at most 5. This implies that, since any three weights are coprime, there must be one weight among $a_2,\ldots,a_{k-1}$ dividing $a_0$ or $a_1$, against our convention on the weights. Thus, we always have
        $$d_1 \geq a_0+a_1+\frac{d_1}{qg}.$$
        Since no two weights among $a_2,\ldots,a_{k-1}$ have a common factor (because $[a_0,a_1]=d_1$), $(d_1/qg;a_2,\ldots,a_{k-1})$ is regular, hence 
        $$\frac{d_1}{qg}-a_2-\ldots-a_{k-1} \geq F(a_2,\ldots,a_{k-1}).$$
        We can now use the fact that $d_1 \geq a_0+a_1+\frac{d_1}{qg}$ to obtain the statement. 
    
    }

        \end{itemize}

    This concludes the proof for $d_1=[a_0,\ldots,a_{k-1}]$. For $d_1>[a_0,\ldots,a_{k-1}]$, the same proofs still work verbatim except when the regularity of $(d''    ;a'')$ is used. But if $d_1 > [a_0,\ldots,a_{k-1}]$, then $d_1 \geq [a_0,\ldots,a_{k-1}]+mg$, where $g=\gcd(a_0,a_1)$ and $m=d_1/[a_0,\ldots,a_{k-1}]$. Then, a similar proof holds by Lemma \ref{lemma:recfrob1g}. For the reader's convenience, we give an example by showing how the case $k=2$ generalises.

    Suppose $g=\gcd(a_0,a_1)>1$ and $d_1>[a_0,a_1]$. Then $mg<d_1$ ($mg=d_1$ is excluded by our convention on the weights, as it corresponds to $a_0=a_1=g$); hence $d_1 \geq [a_0,a_1]+mg$ and
        $$\delta(d;a)> ([a_0,a_1]-a_0-a_1) + F^{m}(a_2,\ldots,a_n,g)+mg.$$
        The result now follows from Lemma \ref{lemma:recfrob1g}.

\end{proof}

Together with Lemma \ref{lemma:reduction} and Proposition \ref{prop:noasterisk}, we obtain the following.

\begin{corollary}
\label{cor:maincor}
For any $h$-regular pair $(d;a)=(d_1,\ldots,d_c;a_0,\ldots,a_n)$ such that $c \leq 3$, $c \leq n$ and $a_i \nmid h$ for all $0 \leq i \leq n$, $\delta(d;a) \geq F^h(a_0,\ldots,a_n)$. In particular, Conjecture \ref{conjecture:frob} holds for $c \leq 3$.
\end{corollary}

\begin{corollary}[=Theorem \ref{thm:main}]
\label{cor:kawamata}
    Let $X\subset \mathbb{P}$ be a well-formed quasi-smooth WCI which is not a linear cone, such that $\textup{codim}X \leq 3$. Let $H$ be an ample Cartier divisor on $X$ such that $H-K_X$ is ample, then $|H| \neq \emptyset$.
\end{corollary}

From the point of view of numerical semigroups, we can interpret the previous results as a bound on Frobenius numbers.
\begin{corollary}

    Let $a_0,\ldots,a_n$ be coprime positive integers, then
    $$F(a_0,\ldots,a_n)\leq \delta(d;a),$$
    where $(d;a)=(d_1,\ldots,d_c;a_0,\ldots,a_n)$ is any regular pair such that $c \leq 3$ and $n \geq c$.
\end{corollary}

\section{About the codimension of minimal pairs}
\label{section:remarks}

Part of the proof of Theorem \ref{thm:codim3} relies on the results of Section \ref{section:reduction} to reduce the number of types of pairs that need to be considered. With the same goal in mind, we spend a couple words on a different type of reduction that could be beneficial to understanding Conjecture \ref{conjecture:h-frob} in higher codimensions.
\\

The statement of Conjecture \ref{conjecture:h-frob} gives a lower bound on $\delta(d;a)$ which is independent of the degrees of the pair: then, in order to prove (or disprove) Conjecture \ref{conjecture:h-frob}, it is sufficient to consider $h$-regular pairs with fixed weights $a_0,\ldots,a_n$ and only check the statement on those pairs that are minimal in some sense. To make things clearer, we restrict ourselves to the case of $h=1$: notice that if Proposition \ref{prop:noasterisk} holds in any codimension, then the case $h>1$ can be ignored for the purposes of Conjecture \ref{conjecture:h-frob} by Proposition \ref{lemma:reduction}.

For any tuple of weights $a=(a_0,\ldots,a_n)$, let 
$$\Delta_a=\{\delta(d;a) \in \mathbb{Z} \mid (d;a)\text{ is a regular pair and }|d|<|a|\}$$ 
be the set of regular pairs with given weights $a_0,\ldots,a_n$ and $c\leq n$. We say that a pair $(d;a)$ such that $\delta(d;a)=\min \Delta_a$ is \emph{minimal} with respect to the weights $a_0,\ldots,a_n$. Note that if $a_i \neq 1$ for all $i$, then a minimal pair satisfies $\delta(d;a)>0$ by Lemma \ref{lemma:deltabound}. In general, it does not seem obvious which properties distinguish a minimal pair from a non-minimal one. A first simple observation is that for a regular pair, all the factors appearing in its degrees are necessary to its regularity. 
\begin{definition}
    A regular pair $(d;a)$ is \emph{reducible} if there exists a degree $d_i$ and a prime $p \mid d_i$ such that the pair $$(d';a')=(d_1,\ldots,d_i/p,\ldots,d_c;a_0,\ldots,a_n)$$
    is again regular. If a pair is not reducible, we say it is \emph{irreducible}.
\end{definition}
A minimal pair is necessarily irreducible: in fact, if a regular pair $(d;a)\in \Delta_a$ is reducible, then the pair$(d';a)$, obtained by replacing a reducible degree $d_i$ with $d_i/p$, satisfies $\delta(d';a) < \delta(d;a)$ and again $\delta(d';a')\in \Delta_a$. 
\\

It is instead less clear whether there is a relation between the minimality of a pair and its codimension. Since, on average, pairs in $\Delta_a$ which have more degrees also have smaller degrees, a naive guess is that a minimal pair in $\Delta_a$ should also have maximal codimension. While we are not able to fully confirm such statement, we point out two facts in support of this idea.

\begin{itemize}
    \item {Suppose that for a set of weights $a=(a_0,\ldots,a_n)$, a minimal pair $(d;a)$ has codimension $c=|d|<n$. Then, there is a pair $(d';a)$ of maximal codimension $c=n$ which is almost minimal: in fact, consider the pair $$(d';a)=(d_1,\ldots,d_c,1,\ldots,1;a_0,\ldots,a_n),$$
    where $d'$ has $n-c$ degrees equal to 1. Then, $\delta(d';a)=\delta(d;a)+n-c.$ Hence, even if $(d';a)$ is not a minimal pair, it is close to being one.}
    
\item{Suppose again that $(d;a)$ is a regular pair with $c<n$. If there is a prime $p$ such that $(d(p);a(p))$ is reducible, then there is a regular pair $(d';a)$ such that $\delta(d';a) < \delta(d;a)$ and $|d'|=c+1$. In fact, suppose that $d_1 \in d(p)$ and that in $d(p)$ it is reducible by some prime $q \neq p$ (that is, if we replace $d_1$ with $d_1/q$ the pair is still regular). Then, the pair $(d';a)=(d_1/p,d_1/q, d_2,\ldots,d_c;a_0,\ldots,a_n)$ is regular and $\delta(d';a) < \delta(d;a)$. To see that $(d';a)$ is still regular, let $g=\gcd(a_{i_1},\ldots,a_{i_k})>1$ for some weights $a_{i_1},\ldots,a_{i_k}$. We only need to check the cases of $p$ or $q$ dividing $g$. First, suppose $p \mid g$. Then, $a_{i_1},\ldots,a_{i_k}$ are weights of the pair $(d(p);a(p))$ as well; since $g$ divides at least $k$ degrees of $d(p)$ by assumption, $g$ divides $k$ degrees among $d_1/q, d_2,\ldots,d_c$. If $q \mid g$ and $p \nmid g$, $g$ still divides $d_1/p$, so it divides at least $k$ degrees among $d_1/g,d_2,\ldots,d_c$. Therefore, $(d';a)$ is also regular. This means that for $c<n$, if there exists some prime $p$ such that the pair $(d(p);a(p))$ is reducible, then $\delta(d;a)$ is not minimal.}
\end{itemize}

Based on the previous observations, we conclude with the following questions.
\begin{question}
\hfill
\begin{itemize}
    
\item{
    For any regular pair $(d;a)$ with $|d|<|a|-1$, is there a prime $p$ such that the pair $(d(p);a(p))$ is reducible?}
\item{If the answer to the previous question is negative, for a given set of weights $a=(a_0,\ldots,a_n)$ is there a minimal pair $(d;a)$ such that $|d|=|a|-1$? }
\end{itemize}
\end{question}

And based on the results of Section \ref{section:reduction}, we ask the following.

\begin{question}
    Are Conjectures \ref{conjecture:frob} and \ref{conjecture:h-frob} equivalent?
\end{question}

\providecommand{\bysame}{\leavevmode\hbox to3em{\hrulefill}\thinspace}
\providecommand{\MR}{\relax\ifhmode\unskip\space\fi MR }
\providecommand{\MRhref}[2]{
  \href{http://www.ams.org/mathscinet-getitem?mr=#1}{#2}
}
\providecommand{\href}[2]{#2}

\end{document}